\documentclass[a4paper,11pt]{amsart}
\usepackage[foot]{amsaddr}
\usepackage[margin=2.7cm]{geometry}
\usepackage{graphicx,color, amsfonts,amsmath}
\usepackage{amssymb}
\usepackage[titletoc]{appendix} 

\usepackage{systeme}
\usepackage{graphicx}
\usepackage[]{units}%
\usepackage{color}%
\usepackage{mathrsfs}
\usepackage{bm}
\usepackage{float}
\usepackage{amsmath,amsfonts}
\usepackage{isomath}
\usepackage{amsthm}
\usepackage{amssymb}
\usepackage{algorithm,algorithmic}
\usepackage{array}
\usepackage{booktabs}
\usepackage{multicol}
\usepackage{multirow}
\usepackage{tabularx}
\usepackage{cite} 

\usepackage{subcaption}
\usepackage{graphicx}
\newtheorem{thm}{Theorem}[section]

\newtheorem{lem}[thm]{Lemma}

\newtheorem{prop}[thm]{Proposition}
\theoremstyle{definition}

\theoremstyle{remark}

\numberwithin{equation}{section}

\newcommand{\vertiii}[1]{{\left\vert\kern-0.25ex\left\vert\kern-0.25ex\left\vert #1 
    \right\vert\kern-0.25ex\right\vert\kern-0.25ex\right\vert}}

\newcommand{\ldb}{\mathopen{\lbrack\!\lbrack}}
\newcommand{\rdb}{\mathclose{\rbrack\!\rbrack}}
\begin{document}
\title[Cut finite element error estimates for a class of nonlinear elliptic PDEs]{Cut finite element error estimates for a class of nonlinear elliptic PDEs}%
\author{Georgios Katsouleas\textsuperscript{1}}
\address{\textsuperscript{1}Department of Mathematics, National Technical University of Athens, Zografou Campus, 15780, Greece.}
\thanks{}
\email{gekats@mail.ntua.gr}

 \author{Efthymios N. Karatzas\textsuperscript{1,2}}
\address{\textsuperscript{2}SISSA (affiliation), International School for Advanced Studies, Mathematics Area, mathLab, Via Bonomea 265, Trieste, 34136, Italy.}
\email{karmakis@math.ntua.gr}
 
\author{Fotios Travlopanos\textsuperscript{1}}
\thanks{}
\email{ftravlo@gmail.com}

\subjclass[2000]{Primary}%
\date{\today} 

\thanks{}%

\begin{abstract}
Motivated by many applications in complex domains with boundaries exposed  to large topological changes or deformations, fictitious domain methods regard the actual domain of interest  as being embedded in a fixed Cartesian background. This is usually achieved via a geometric parameterization of its boundary via level--set functions.
In this note, the a priori analysis of unfitted numerical schemes with cut elements is extended beyond the realm of linear problems. 
More precisely, we consider the discretization of 
 semilinear elliptic boundary value problems  of the form $- \Delta  u +f_1(u)=f_2$ with polynomial nonlinearity via the cut finite element method.  Boundary conditions are enforced, using a Nitsche--type approach. To ensure stability and  error estimates that are independent of the position of the boundary with respect to the mesh, the formulations are augmented with additional boundary zone ghost penalty terms. These terms act on the jumps of the normal gradients at faces associated with cut elements. A--priori error estimates are derived, while
numerical  examples  illustrate the implementation of the method and validate the theoretical findings.    
\end{abstract}

\maketitle

\section{Introduction}

Fictitious domain methods have a long history, dating back to the pioneering  work of Peskin \cite{P1972} and are currently enjoying great popularity, having been successfully applied to a variety of problems. Several variants include such methods as the ghost-cell finite difference method \cite{WFC13}, cut--cell volume method \cite{PHO16}, immersed interface \cite{KB18}, ghost fluid \cite{BG14},  shifted boundary methods \cite{ACS20,  KSASRa2019, MS18}, $\phi$--FEM \cite{DL19}, and CutFEM \cite{BCHLM2014,BH2010,BH2012, L19, HaHa02}, among others.
For a comprehensive overview of this research area, the interested reader is referred to the review paper \cite{MI05}. Considerable impetus has been provided  in the contexts of fluid--structure interaction and reduced order modeling for parametrically--dependent domains \cite{KBR19,KSNSRa2019}. 

Such cases pose severe challenges in the discretization and even result to  simulations of diminished quality. For instance, the generation of  a suitable conforming mesh is a challenging and computationally intensive task. As a means to bypass such complications, it is instructive to consider the actual computational domain of interest  as being embedded in an unfitted background mesh. More precisely, this can be achieved via a geometric parameterization of its boundary via level--set geometries, using a fixed Cartesian background and its associated mesh for each new domain configuration. This approach avoids the need to remesh, as well as the need to develop a reference domain formulation. In such cases, immersed and embedded methods compare favorably to fitted mesh FEMs, providing simple and efficient schemes for the numerical approximation of PDEs in both cases of static and evolving geometries.

The overall objective of this note  is to extend the a--priori analysis  of cutFEM beyond the realm of linear problems. To this end, we propose an unfitted framework for the numerical solution of a semilinear elliptic boundary value problem  with a polynomial nonlinearity. We start by introducing the model problem and the necessary notation in Section \ref{Section2}. Then, Section \ref{Section3} focuses on the derivation of the a--priori error estimates and a numerical experiment is reported in Section \ref{Section4}, verifying the theoretical convergence rates and showcasing the accuracy of the method. The paper concludes with a brief discussion  of our contributions and suggestions for future work in Section \ref{Section5}.

\section{The model  problem and preliminaries} \label{Section2}  

As a model problem, we consider a semilinear elliptic boundary value problem of the form 
 \begin{eqnarray} \label{Semilinear Problem} 
 - \Delta  u +f_1(u)&=&f_2 \quad\,\,\,\,\, \text{in $\Omega$}, \\
u&=&0   \qquad \text{ on $\Gamma$}, \nonumber
\end{eqnarray}
where $\Omega \subset \mathbb{R}^2$ is a simply connected open domain with  boundary $\Gamma =\partial \Omega$. The nonlinearity is assumed to be of polynomial type $f_1(u)=\left|u\right|^{p-2}u$. Such equations have been studied previously in the context of problems with critical exponents \cite{Clement:1996}  and are referred to  in the theory of boundary layers of viscous fluids \cite{W1975} as Emden--Fowler equations. It is straightforward to verify that the weak formulation 
\begin{equation}\label{weak_sol}
\int_{\Omega}\left(\nabla u \cdot \nabla v + f_1(u)v\right)=\int_{\Omega}  f_2v, \ \ \text{for every } \ v \in H_0^1(\Omega)
\end{equation}
of (\ref{Semilinear Problem}) admits a weak solution $u \in H_0^1(\Omega)$. Following a standard energy argument and assuming the force $f_2 \in H^{-1}(\Omega)$, the a--priori error bound 
$$
\frac{1}{2}\left\|\nabla u\right\|^2_{L^2(\Omega)}+\left\|u\right\|^p_{L^p(\Omega)}\leq \frac{1}{2}\left\|f_2\right\|^2_{H^{-1}(\Omega)}
$$
readily follows, indicating a continuous dependence of the solution on the data.
  
Implementation of an unfitted FEM for the discretization of  (\ref{weak_sol}) requires a fixed background domain $\mathcal{B}$ which contains $\Omega$; let  $\mathcal{B}_h$ its corresponding shape--regular mesh. The  \itshape active \normalfont mesh 
$$
\mathcal{T}_h=\left\{T\in \mathcal{B}_h: T\cap \Omega\neq \emptyset\right\}
$$  
is the  minimal submesh of $ \mathcal{B}_h$ which covers $\Omega$ and is in general \itshape unfitted \normalfont to its boundary $\Gamma$. As usual, the subscript $h=\max_{T\in  \mathcal{B}_h}diam(T)$ indicates the global mesh size. The finite element space for discrete solutions will in fact be built upon the \itshape extended \normalfont domain $\Omega_{\mathcal{T}}=\bigcup_{T \in \mathcal{T}_h}T$ which corresponds to  $\mathcal{T}_h$.  Fictitious domain  methods  require boundary conditions at $\Gamma$ to be weakly satisfied through a variant of Nitsche's method. On the other hand, coercivity over the whole computational domain $\Omega_{\mathcal{T}}$  is ensured by means of additional ghost penalty terms which act on the gradient jumps in the boundary zone; see, for instance, \cite{BH2012}. Therefore, a more detailed analysis of the interface grid is required; the submesh consisting of all cut elements  is denoted
$$
{G}_h:=\{T\in \mathcal{T}_h:T\cap\Gamma\neq\emptyset\}
$$
and the relevant set of faces upon which ghost penalty will be applied is given by
$$
\mathcal{F}_{G}:=\left\{F: F \text{ is a  face of } T\in G_h, F\notin \partial\Omega_{\mathcal{T}}\right\}.
$$

Considering the finite element space 
$$
 V_h:=\left\{w_h\in C^0(\bar{\Omega}_T): w_h|_{T}\in \mathcal{P}^1(T), T \in \mathcal {T}_h\right\}
$$
for approximate solutions, we define discrete counterparts to the continuous bilinear and linear forms in (\ref{weak_sol}), setting
\begin{align}
a_h(u_h,v_h)&=\int_{\Omega} \nabla u_h \cdot \nabla v_h -\int_{\Gamma_D}v_h\left({\bf n_\Gamma}\cdot \nabla u_h\right)-\int_{\Gamma_D}u_h\left({\bf n_\Gamma}\cdot \nabla v_h\right)+\gamma_Dh^{-1}\int_{\Gamma_D}u_hv_h, \label{ah}\\
\ell_h(v_h)&=\int_{\Omega}f_2 v_h \label{ellh}
\end{align}
for $u_h, v_h \in V_h$. Here, ${\bf n_{\Gamma}}$ denotes the outward pointing unit normal vector on the boundary  $\Gamma$. The cutFEM discretization scheme reads as follows: find a discrete state $u_h \in V_h$, such that 
\begin{equation}\label{Discr_state}
a_h(u_h,v_h)+ { j_h(u_h,v_h)  }+\int_{\Omega}f_1(u_h) v_h=\ell_h(v_h),
\end{equation}
for all $v_h \in V_h$, where the stabilization term
 \begin{equation} \label{Stabilization term}
  j_h(u_h,v_h)=\sum_{F\in \mathcal{F}_G}\gamma_1 h\int_{F}\ldb{{\bf n_F}}\cdot\nabla  u_h\rdb\ldb{{{\bf n_F}}\cdot\nabla} v_h\rdb,
  \end{equation}
acts on the gradient jumps $ \ldb{{\bf n_F}}\cdot\nabla  u_h\rdb:={\bf n_F}\cdot\nabla  u_h \Bigr|_{K}-{\bf n_F}\cdot\nabla  u_h \Bigr|_{K^{'}}$of $u_h$  over element faces $F=K\cap K^{'}$ in the interface zone and  is  included in the bilinear form to extend its coercivity from the physical domain $\Omega$ to $\Omega _{\mathcal{T}}$. The quantities $\gamma_D$ 
and $\gamma_1$  in (\ref{ah}) and (\ref{Stabilization term})  are  positive penalty parameters; see Lemma \ref{stab} below.

\section{Norms, approximation properties and a--priori analysis} \label{Section3} 

The convergence analysis of the method (\ref{Discr_state}) is based on the following mesh--dependent norms:
  \begin{eqnarray}
& \vertiii{v}_{*}^2 = \left\|\nabla v\right\|^2_{L^2(\Omega)}+ \left\|h^{-1/2}\gamma_{D}^{1/2} v\right\|^2_{L^2(\Gamma)},\nonumber  \\    
& \vertiii{v}_{h}^2 = \left\|\nabla v\right\|^2_{L^2(\Omega_{\mathcal{T}})}+ \left\|h^{-1/2}\gamma_{D}^{1/2} v\right\|^2_{L^2(\Gamma)}+j_h(v,v).\nonumber
 \end{eqnarray}
The trace inequality  $\left\|v\right\|_{L^2(T \cap \Gamma)} \leq C_{tr}\left(h_T^{-1/2}\left\|v\right\|_{L^2(T)}+h_T^{1/2}\left\|\nabla v\right\|_{L^2(T)}\right)$ for $ T \in \mathcal{T}_h$ and $h_T=diam(T)$ implies in particular: $$\vertiii{v_h}_{*}\leq C_{*} \vertiii{v_h}_h.$$
A necessary approximation result is stated next:

\begin{lem}[\cite{BH2012}, Lemma 5]\label{approx}
Let $\mathcal{E}: H^2(\Omega) \rightarrow H^2(\Omega_{\mathcal{T}})$ a linear $H^2$--extension operator   on ${\Omega}_{\mathcal T}$, such that $\mathcal{E}\phi|_{\Omega} =\phi|_{\Omega}$, $\mathcal{E}\phi|_\Gamma =\phi|_\Gamma$, $\left\|\mathcal{E}\phi\right\|_{H^2(\Omega_{\mathcal{T}})}\lesssim \left\|\phi\right\|_{H^2(\Omega)}$ and $\Pi_h: H^1(\Omega) \rightarrow V_h$ the Cl{\'e}ment-type extended interpolation operator  defined by  $\Pi_h \phi= \Pi^*_h\mathcal{E} \phi$, where $\Pi_h^*: H^1(\Omega_{\mathcal{T}}) \rightarrow V_h$ is the standard Cl{\'e}ment interpolant.  Then, the estimate
\begin{equation}\label{approx2}
\vertiii{u-\Pi_h u}_{*}+j(\Pi_h u, \Pi_h u)^{1/2}\leq Ch |u|_{H^2(\Omega)}  
\end{equation}
holds for every $u \in H^2(\Omega)$.
\end{lem}

Regarding stability, the coercivity and continuity properties of the augmented bilinear form $\left[a_h+j_h\right](\cdot, \cdot)$ now read as follows:

\begin{lem}[\cite{BH2012}, Lemmata 6 and 7]\label{stab}
Defining the method \normalfont (\ref{Discr_state}) \itshape with sufficiently large parameter $\gamma_D$   and $\gamma_1=1$, then 
\begin{equation}\label{coerc}
 c_{bil}\vertiii{u_h}_h^2\leq a_h(u_h,u_h)+j_h(u_h,u_h), \quad a_h(u_h,v_h)+j_h(u_h,v_h) \leq  C_{bil}\vertiii{u_h}_h \vertiii{v_h}_h,
\end{equation}
for every $u_h, v_h \in V_h$, and 
\begin{equation}\label{cont}
a_{h}(v,v_h)\leq c_a\vertiii{v}_{*}\vertiii{v_h}_{*}, \quad \text{ for every } v \in \left[H^2(\Omega)+V_h\right] \text{ and } v_h \in V_h,  
\end{equation}
independently of $h$ and of the way in which the boundary $\Gamma$ intersects the background mesh.
\end{lem}

Hence, due to the gradient penalty in the boundary zone, control of the $L^2$--norm of the gradient may be extended  over the whole active mesh $\mathcal{T}_h$.  

We next quantify how the additional  term $j_h(u_h, v_h)$ affects the Galerkin orthogonality and consistency of the variational formulation (\ref{Discr_state}). 

\begin{lem}[Galerkin orthogonality]\label{Galerkin orthogonality}
Let $u \in H^1_0(\Omega)$ be the solution to the semilinear problem \normalfont (\ref{weak_sol}) \itshape and $u_h \in V_h$ its finite element approximation in  \normalfont  (\ref{Discr_state}).  \itshape Then, 
\begin{equation}\label{galerkin}
a_h(u_h-u, v_h) = \int_{\Omega}\left[f_1(u)-f_1(u_h)\right]v_h-j_h(u_h, v_h), \quad \text{for every} \quad  v_h \in V_h.
  \end{equation} 
\end{lem}
\begin{proof}
Recalling the definitions of $a_h$ and $\ell_h$ in (\ref{ah}) -- (\ref{ellh}), it is immediate that the solution $u$ satisfies the equation $a_h(u, v_h)+\int_{\Omega}f_1(u)v_h=\ell_h(v_h)$, for every $v_h \in V_h$ and the result follows.  
\end{proof}

The following preliminary result investigating optimality with respect to interpolation is a  key ingredient of our approach.

\begin{prop}\label{pre}
Let $u \in H^1_0(\Omega)$ be the solution to the semilinear problem \normalfont (\ref{weak_sol}) \itshape and $u_h \in V_h$ its finite element approximation in  \normalfont  (\ref{Discr_state}).  \itshape Then, there exists a constant $C>0$, independent of $u$, $u_h$, such that 
\begin{equation}\label{error}
\vertiii{u_h-\Pi_hu}_h^2+\left\|u-u_h\right\|^p_{L^p(\Omega)}\leq C \left( \left[\vertiii{u-\Pi_h u}_{h} + j_h(\Pi_hu, \Pi_hu)^{1/2}\right]^2+\left\|u-\Pi_hu\right\|_{L^q(\Omega)}^q \right),
\end{equation} 
where $q$ is the  conjugate index of the power $p$ in the nonlinear term $f_1(u)=\left|u\right|^{p-2}u$.
\end{prop}
\begin{proof} 
Adapting for our purposes  the procedure in the proof of  \cite[Thm. 5.3.3, p. 319]{Ciarlet:1978} for the $p$--Laplacian, a first observation is that there exists $c>0$, such that
\begin{equation}\label{lemm}
\int_{\Omega}f_1(u-u_h)(u-u_h)\leq c\int_{\Omega}\left[f_1(u)-f_1(u_h)\right](u-u_h).
\end{equation}
Then, denoting $e_h:=u_h-\Pi_h u$, we successively  apply  the coercivity estimate (\ref{coerc}), (\ref{lemm}) and  the Galerkin orthogonality (\ref{galerkin}) to estimate 
\begin{align*}
 c_{bil}\vertiii{e_h}_{h}^2 +\frac{1}{c}\left\|u-u_h\right\|_{L^p(\Omega)}^p& \leq \left[a_h+j_h\right](e_h,e_h)+\frac{1}{c}\int_{\Omega}f_1(u-u_h)(u-u_h) \\
 & = a_h(u-\Pi_h u,e_h)+a_h(u_h- u,e_h) +j_h(e_h,e_h) + \\
 & \quad  + \int_{\Omega}\left[f_1(u)-f_1(u_h)\right](u-u_h) \\
 & = a_h(u-\Pi_h u,e_h)+j_h(-\Pi_h u, e_h)  +\int_{\Omega}\left[f_1(u)-f_1(u_h)\right](u-\Pi_hu).
\end{align*} 
A bound for the leading two terms is readily implied by   the continuity estimate (\ref{cont}), the Cauchy--Schwarz inequality and (\ref{approx2}): \small
\begin{align*}
 a_h(u-\Pi_h u,e_h)+j_h(-\Pi_h u, e_h)  &\leq c_a \vertiii{u-\Pi_h u}_{*}\vertiii{e_h}_{*}  + j_h(\Pi_hu, \Pi_hu)^{1/2}j_h(e_h,e_h)^{1/2} \\
 &\leq \left[c_aC_{*} \vertiii{u-\Pi_h u}_{*} + j_h(\Pi_hu, \Pi_hu)^{1/2}\right]\vertiii{e_h}_{h} \\
 &\leq \frac{\max\left\{c_aC_{*}, 1\right\}^2}{2c_{bil}}\left[\vertiii{u-\Pi_h u}_{*} + j_h(\Pi_hu, \Pi_hu)^{1/2}\right]^2+\frac{c_{bil}}{2}\vertiii{e_h}_{h}^2, 
\end{align*}\normalsize
while the third term is estimated by
\begin{align*}
\int_{\Omega}\left[f_1(u)-f_1(u_h)\right](u-\Pi_hu) &\leq C_{f_1} \left\|u-u_h\right\|_{L^p(\Omega)}\left\|u-\Pi_hu\right\|_{L^q(\Omega)} \\
&\leq \frac{1}{2c}\left\|u-u_h\right\|_{L^p(\Omega)}^p+\left(\frac{p}{2c}\right)^{-q/p}\frac{C_{f_1}}{q}\left\|u-\Pi_hu\right\|_{L^q(\Omega)}^q.
\end{align*}\normalsize
Hence, the  assertion  (\ref{error}) already follows for \small $C=\min\left\{\frac{c_{bil}}{2}, \frac{1}{2c} \right\}^{-1}\max\left\{\frac{\max\left\{c_aC_{*}, 1\right\}^2}{2c_{bil}}, \left(\frac{p}{2c}\right)^{-q/p}\frac{C_{f_1}}{q}\right\}$.  \normalsize 
\end{proof}

Under some additional regularity requirements for the solution $u$, we are now in a position to derive error estimates for our finite element approximations: 

\begin{thm}[Optimal convergence]\label{convergence}
Let $u \in H^1_0(\Omega)\cap H^2(\Omega)\cap  W^{2,q}(\Omega)$ be the solution to the semilinear problem \normalfont (\ref{weak_sol}) \itshape and $u_h \in V_h$ its finite element approximation in  \normalfont  (\ref{Discr_state}).  \itshape Then,  $\vertiii{u-u_h}_{*} =\mathcal{O}(h)$.
\end{thm}
\begin{proof}
We first decompose the total error $\vertiii{u-u_h}_{*}$ into its discrete--error and projection--error components; i.e.,
 $$
\vertiii{u-u_h}_{*}\leq \vertiii{u-\Pi_hu}_{*}+C_{*}\vertiii{\Pi_h u-u_h}_{h}.
 $$
 The desired estimate for the first term is already provided by (\ref{approx2}). Hence, it suffices to prove the assertion for  the latter, which is in turn  bounded by Proposition \ref{pre}. Indeed, by (\ref{approx2}) and the properties of the Cl{\'e}ment interpolant \cite[p.69]{EG2004}, estimate (\ref{error}) yields 
$$
\vertiii{u_h-\Pi_hu}_h^2 \leq \hat{C}  \left(h^2\left|u\right|_{H^2(\Omega)}^2+h^{2q}\left|u\right|^q_{W^{2,q}(\Omega)}\right)
$$ 
for $\hat{C}>0$. Recalling $q=\frac{p}{p-1}$ is the conjugate index of $p$, clearly $\min\left\{1, q\right\}=1$ and the bound is optimal. 
\end{proof}

\section{Numerical validation}\label{Section4}

In order to  verify the validity of the a-priori error estimate in Theorem \ref{convergence}, numerical simulations have been implemented in a python environment, using the open--source Netgen/NGSolve-ngsxfem finite element software. We consider a two--dimensional test case of (\ref{Semilinear Problem})  for $p=4$ with manufactured exact solution and right--hand side force defined by
$$
u\left(x,y\right)=\frac{1}{2}(1-x^2-y^2), \quad f\left(x,y\right)=\frac{1}{8}\left(1-x^2-y^2\right)^3+2
$$
in $\Omega=\mathcal{D}(0,1)$; i.e., the unit disc centered at the origin. As in Section \ref{Section2}, 
 the original domain $\Omega$ is immersed in the background domain $\mathcal{B}=\left[-1.5, 1.5\right]^{2}$.  To investigate orders of convergence, we consider a sequence of successively refined tessellations $\{\mathcal{B}_{h_\ell}\}_{\ell\geq 0}$ for $\mathcal{B}$ with
mesh parameters $h_\ell=0.15 \times 2^{-\ell}$ ($\ell=0, \dots ,6$). The stabilization constants  $\gamma_D$, $\gamma_1$ in (\ref{ah}) and (\ref{Stabilization term}) are taken to be equal to   $\gamma_D=1$ and $\gamma_1= 0.1$ respectively. By the theoretical error estimate stated in Theorem \ref{convergence}, we should expect first--order convergence rates with respect to the $H^{1}$--norm and additionally second order for the $L^{2}$--norm.

\begin{table}[H] 
\caption{Errors and experimental orders of convergence (EOC) for  $H^1$ and $L^2$ norms.}\label{table1}  
 \begin{center}
    \begin{tabular}{lcccc}
\toprule
$h_{\max}$ & $\qquad \left\|u-u_h\right\|_{H^1(\Omega)}$ &  $\text{EOC}$ & $\qquad \left\|u-u_h\right\|_{L^2(\Omega)}$ & $\text{EOC}$   \\ \midrule 
$0.15$ & 7.74620e-2 &  & 2.47468e-3  &     \\ 
$0.075$   & 3.90601e-2 &0.988 & 5.83351e-4  &  2.085  \\ 
$0.0375$     & 1.93383e-2 & 1.014 & 1.33451e-4  &   2.128 \\ 
$0.01875$    & 9.63082e-3 & 1.006 & 3.34143e-5  &  1.999 \\ 
$0.009375$    & 4.80627e-3 & 1.003 & 8.12293e-6  &  2.040 \\ 
$0.0046875$   & 2.40450e-3 & 0.999 &2.01406e-6  &   2.012 \\  
\midrule
Mean       &          & 1.002 &      & 2.049 \\
\bottomrule
\end{tabular}
\end{center}
\end{table}

As illustrated in Table \ref{table1}, the numerical findings validate the theoretically predicted rates of convergence and verify the effectiveness of the proposed framework.

\section{Conclusions}\label{Section5}

The present note concentrated on the derivation of  an a--priori error estimate for a cut finite element approximation of a semilinear model problem. To the authors' best knowledge, this is one of the few instances in the literature that such an analysis has been carried out beyond a linear context. Our approach is based on classical arguments for the $p$--Laplacian  \cite{Ciarlet:1978} and on key results from \cite{BH2012}  for a stabilized unfitted method for the Poisson problem.  
Future work will delve more deeply in the analysis of unfitted FEMs for general  time--dependent  problems with nonlinearities. From a computational point of view, the effect of preconditioning on the performance of the method  will be assessed in the spirit of \cite{ArKar19, KKA20}. Finally, the method seems promising for  controlling nonlinear PDEs  with uncertainties, involving large deformations and/or topological changes \cite{KBR19, KSNSRa2019, KSNSRa2019sub}.

\section*{Acknowledgments}
 
 This project has received funding from the Hellenic Foundation for Research and Innovation (HFRI) and  the  General  Secretariat  for  Research  and  Technology (GSRT), under  grant agreement No[1115] (PI: E. Karatzas), and the support  of the National Infrastructures for Research and Technology S.A. (GRNET S.A.) in the National HPC facility - ARIS - under project ID pa190902.

\end{document}